\documentclass[12pt,draft]{amsart}

\usepackage[all]{xy}
\usepackage{amsfonts}
\usepackage{amssymb}

\usepackage{enumerate}

\usepackage{color}

\newtheorem{teo}{Theorem}[section]
\newtheorem{lem}[teo]{Lemma}
\newtheorem{prop}[teo]{Proposition}

\newtheorem{cor}[teo]{Corollary}

\theoremstyle{definition}
\newtheorem{dfn}[teo]{Definition}
\newtheorem{rk}[teo]{Remark}
\newtheorem{ex}[teo]{Example}

\def\<{\langle}
\def\>{\rangle}
\def\ss{\subset}
\def\sse{\subseteq}

\def\a{\alpha}

\def\e{\varepsilon}
\def\l{{\lambda}}

\def\f{{\varphi}}

\def\F{{\Phi}}

\def\C{{\mathbb C}}

\def\A{{\mathcal A}}

\def\M{{\mathcal M}}
\def\cN{{\mathcal N}}
\def\cZ{{\mathcal Z}}

\def\1{\mathbf 1}

\def\N{{\mathbb N}}

\begin{document}
\title[Uniform structure on  Hilbert $C^*$-modules]
{A new uniform structure for  Hilbert $C^*$-modules}
\author{Denis Fufaev}
\address{Moscow Center for Fundamental and Applied Mathematics,
Dept. of Mech. and Math., 
	Lomonosov Moscow State University, 119991 Moscow, Russia}
\email{fufaevdv@rambler.ru, denis.fufaev@math.msu.ru}
\author{Evgenij Troitsky}
\thanks{This work is
supported by the Russian Science Foundation
under grant 23-21-00097.}
\address{Moscow Center for Fundamental and Applied Mathematics,
Dept. of Mech. and Math., 
	Lomonosov Moscow State University, 119991 Moscow, Russia}
\email{evgenij.troitsky@math.msu.ru}



\begin{abstract}
We introduce and study some new uniform structures for Hilbert $C^*$-modules over 
an algebra $\A$. In particular, we prove that in some cases they have the same totally bounded 
sets. To define one of them, we introduce a new class of $\A$-functionals: locally adjointable functionals, which have interesting properties in this context and seem to be of independent interest.
A relation between these uniform structures and the theory of $\A$-compact operators is established.
\end{abstract}

\maketitle

\section*{Introduction}
In the theory of Hilbert $C^*$-modules there are problems in which the necessity to construct uniform structures arises naturally. More precisely, this is the case for the theory of $\A$-compact operators. In the case of Hilbert spaces, i.e. in the case $\A=\C$, the geometric description of such operators is well known: the operator is compact if and only if the image of the unit ball is totally bounded in norm. In general, this is not true for Hilbert $C^*$-modules: even if we take any infinite-dimensional unital $C^*$-algebra as a module over itself and the identity operator, it is $\A$-compact (it has $\A$-rank one), but the unit ball is not totally bounded due to infinite dimension.
Therefore, to describe the $\A$-compactness property in geometric terms, it is necessary to construct a new geometric structure on the Hilbert $C^*$-module, for example, a uniform structure, i.e. a system of pseudometrics or seminorms.
This problem was considered to be unsolvable in reasonable generality for a long time.
Only partial advances were obtained in \cite{lazovic2018,KeckicLazovic2018}. 
Nevertheless in \cite{Troitsky2020JMAA} a uniform structure was discovered that gave a solution in the case of any algebra and any countably generated module as the range module of the operator under consideration. Namely, if $F:\M\to\cN$ is an adjointable operator and $\cN$ is countably generated then $F$ is $\A$-compact if and only if the image of the unit ball is totally bounded with respect to each defining seminorm for the uniform structure. In \cite{TroitFuf2020} the result was strengthened: the necessity of the condition was established for arbitrary modules, the sufficiency was established for modules with some analogue of the projectivity property (it turns out that this property is equivalent to the existence of a standard frame). However, by using this uniform structure the problem cannot be completely solved. In particular, in \cite{Fuf2021faa} a counterexample was constructed: a specific $C^*$-algebra, considered as a module over itself, for which the identity operator is not $\A $-compact, but the unit ball is totally bounded with respect to the introduced uniform structure (in \cite{Fuf2022Path} this work was continued with a close relation
to the theory of frames).
This close connection with the theory of frames has its origin in the fact that Bessel sequences 
in the module context are involved in the construction of the above seminorms. 

The attempts to solve the above problem in full generality lead to the problem of search for more general uniform structures analogous to that considered in above papers.
The idea is to take in the definition of a Bessel sequence elements not from the module itself, but from some larger module. In particular, it is possible to replace elements of the module by $\A$-linear functionals.

In the present paper we introduce some new uniform structures constructed in this way and establish that in some cases they have the same totally bounded sets as the old uniform structure
\cite{Troitsky2020JMAA}.

We also define a new class of $\A$-functionals, slightly more general than the class of adjointable functionals --- locally adjointable functionals. By their properties they are similar to left multipliers, but in some cases they can be described simply in terms of multipliers.

In Section 1 we first recall some facts about $C^*$-algebras and Hilbert $C^*$-modules which we need. Then introduce new uniform structures which generalize the old one in a natural way. Also we obtain some useful properties (Lemma \ref{12then3}, Lemma \ref{new_adm}, Lemma \ref{bounded}). In particular, we prove, that boundedness with respect to any of these uniform structures implies boundedness in norm (that is not true typically for uniform structures, for example, for weak topology).

In Section 2 we deal with the uniform structure which is constructed via  multipliers, and prove that any set is totally bounded with respect to it if and only if it is totally bounded with respect to the old one. This result holds for arbitrary module $\cN$.

In Section 3 we work with the uniform structure which is constructed using a more general class of functionals, the locally adjointable functionals, and prove a similar result but only for standard and countably generated modules.
It turns out that the results on the structure of functionals on the standard module, obtained in \cite{Bak2019}, as well as the Kasparov stabilization theorem, which allows us to reduce the problem to the case of the standard module, play a significant role here.

\section{Preliminaries and formulation of results}
We start with several statements about states on $C^*$-algebras. 

\begin{lem}\label{lem:comparkvadr} \cite[Theorem 3.3.2]{Murphy}
For any state $\f$ on $\A$ and any $a\in\A$ one has
$|\f(a)|^2\le \f(a^*a)$.
\end{lem}

\begin{lem}\label{lem:estimfornonpositive} \cite[Lemma 1.2]{Troitsky2020JMAA}
For any $a\in \A$ there is a state $\f$ such that
$\|a\|\le 2 |\f(a)|.$  
\end{lem}

\begin{lem}\label{uniform} \cite[Lemma 2.1]{Fuf2023}
Let $\f$ be an arbitrary state on $C^*$-algebra $\A$, $\{e_\l\}$ --- an approximate identity in $\A$. Then 
$\f(x-e_\l x)\to 0$ uniformly on bounded sets.
Moreover, for any $n\in \N$ there exists a positive $g_n\in\A$, $||g_n||\le1$, such that $|\f(x)-\f(g_nx)|\le\frac{||x||}{n}$ for any $x\in\A$. 
\end{lem}

Also we need some basic facts about Hilbert
$C^*$-modules over $\A$ and $\A$-compact operators in them. 
One can find 
details and proofs
in books \cite{Lance,MTBook} and 
the survey paper \cite{ManuilovTroit2000JMS}. Some other
directions joining Hilbert $C^*$-modules and operator
theory can be found in
 \cite{FMT2010Studia,PavlovTro2011,BlanchGogi,TroManAlg}.

\begin{dfn}
A (right) pre-Hilbert $C^*$-module over a $C^*$-algebra $\A$
is an $\A$-module equipped with an $\A$-\emph{inner product}
$\<.,.\>:\M\times\M\to \A$ being a sesquilinear form on the
underlying linear space and restricted to satisfy:
\begin{enumerate}
\item $\<x,x\> \ge 0$ for any $x\in\M$;
\item $\<x,x\> = 0$ if and only if $x=0$;
\item $\<y,x\>=\<x,y\>^*$ for any $x,y\in\M$;
\item $\<x,y\cdot a\>=\<x,y\>a$ for any $x,y\in\M$, $a\in\A$.
\end{enumerate}

A pre-Hilbert $C^*$-module over $\A$ is a 
\emph{Hilbert $C^*$-module}
if it is complete w.r.t. its norm $\|x\|=\|\<x,x\>\|^{1/2}$.

A Hilbert $C^*$-module $\M$ is \emph{countably generated}
if there exists a countable set of its elements with dense
set of $\A$-linear combinations.

We will denote by $\oplus$ the Hilbert sum of Hilbert
$C^*$-modules in an evident sense.
\end{dfn}

We have the following Cauchy-Schwartz inequality \cite{Pas1}
(see also \cite[Proposition 1.2.4]{MTBook}): for any $x,y\in\M$,
\begin{equation}\label{eq:cau_schw}
\<x,y\>\<y,x\>\le \|y\|^2 \<x,x\>.
\end{equation}

\begin{dfn}\label{dfn:standard_hm}
The \emph{standard} Hilbert $C^*$-module $\ell^2(\A)$
(also denoted by $H_\A$) is the set of all infinite sequences
$a=(a_1,a_2,\dots)$, $a_i\in\A$, such that
the series $\sum_i (a_i)^* a_i$ is norm-convergent in $\A$.
It is equipped with the inner product $\<a,b\>=
\sum_i (a_i)^*b_i$, where $b=(b_1,b_2,\dots)$.

If $\A$ is unital, then $\ell^2(\A)$ is countably
generated.
\end{dfn}

One of the most nice properties of countably generated
modules is the following theorem \cite{Kasp} (see
\cite[Theorem 1.4.2]{MTBook}). We mean that an isomorphism
preserves the $C^*$-Hilbert structure.

\begin{teo}[Kasparov stabilization theorem]\label{teo:Kasp_st}
For any countably generated Hilbert $C^*$-module $\M$
over any algebra $\A$,
there exists an isomorphism of Hilbert $C^*$-modules
$\M\oplus \ell^2(\A)\cong \ell^2(\A)$.  
\end{teo}

Now recall the main concepts and results of \cite{Troitsky2020JMAA}.

\begin{dfn}\label{dfn:admissyst}
Let $\cN$ be a Hilbert $C^*$-module over $\A$. A countable
system $X=\{x_{i}\}$ of its elements 
is called \emph{admissible} for a submodule $\cN_0\subseteq  \cN$
(or $\cN^0$-admissible) if
 for each $x\in\cN^0$
partial sums of the series $\sum_i \<x,x_i\>\<x_i,x\>$ 
are bounded by $\<x,x\>$ 
and the series is convergent. 
Also we require $\|x_i\|\le 1$ for any $i$.
\end{dfn}

\begin{ex}\label{ex:firstex}
For the standard module $\ell^2(\A)$ over a unital
algebra $\A$ one can take for $X$ the natural base $\{e_i\}$.
In the case of $\ell^2(\A)$ over a general algebra $\A$,
one can take $x_i$ having only the $i$-th component nontrivial
and of norm $\le 1$.
The other important example is $X$ with only finitely many non-zero
elements  and an appropriate normalization (this works for any module).
\end{ex}

Denote by $\F$ a countable collection $\{\f_1,\f_2,\dots\}$
of states on $\A$. For each pair $(X,\F)$ 
with an $\cN^0$-admissible $X$,
consider the following seminorms

\begin{equation}\label{eq:defnnorm}
\nu_{X,\F}(x)^2:=\sup_k 
\sum_{i=k}^\infty |\f_k\left(\<x,x_i\>\right)|^2,\quad x\in \cN^0. 
\end{equation}

and corresponding pseudo-metrics
\begin{equation}\label{eq:defnmetr}
d_{X,\F}(x,y)^2:=\sup_k 
\sum_{i=k}^\infty |\f_k\left(\<x-y,x_i\>\right)|^2,\quad x,y\in \cN^0. 
\end{equation}

In \cite{Troitsky2020JMAA} it is proved that these pseudo-metrics define a uniform structure
and the following definition is introduced.

\begin{dfn}\label{dfn:totbaundset}
A set $Y\subseteq \cN^0 \subseteq  \cN$ is \emph{totally bounded}
with respect to this uniform structure, if for any $(X,\F)$, 
where $X \subseteq  \cN$ is $\cN^0$-admissible,
and any 
$\e>0$ there exists a finite collection $y_1,\dots,y_n$
of elements of $Y$ such that the sets
$$
\left\{ y\in Y\,|\, d_{X,\F}(y_i,y)<\e\right\}
$$  
form a cover of $Y$. This finite collection is an 
$\e$\emph{-net in $Y$ for} $d_{X,\F}$.

If so, we will say briefly that $Y$ is $(\cN,\cN^0)$-\emph{totally bounded}.
\end{dfn}

The main result of \cite{Troitsky2020JMAA} is the following theorem.

\begin{teo}\label{teo:mainteoremold}
Suppose that $F:\M\to\cN$ is an adjointable operator and
$\cN$ is countably generated. Then $F$ is $\A$-compact
if and only if $F(B)$ is $(\cN,\cN)$-totally bounded,
where $B$ is the unit ball of $\M$.
\end{teo}

In the present paper we consider a natural generalization of that uniform structure,

\begin{dfn}\label{dfn:admissyst*}
Let $\cN$ be a Hilbert $C^*$-module over $\A$. A countable
system $F=\{f_{i}\}$ of elements of the dual module $\cN'$ (i.e. $\A$-linear maps $\cN\to \A$) 
is called \emph{$*$-admissible} for a submodule $\cN_0\subseteq  \cN$
(or $*$-$\cN^0$-admissible) if
\begin{enumerate}
\item[1)]  for each $x\in\cN^0$ partial sums of the series $\sum_i (f_i(x))^* f_i(x)$ 
are bounded by $\<x,x\>$ ;
\item[2)] this series is norm convergent; 
\item[3)]  $\|f_i\|\le 1$ for any $i$. 
\end{enumerate}
\end{dfn}

\begin{lem}\label{12then3} 
If $\cN^0=\cN$, then condition 3) follows from conditions 1) and 2).
\end{lem}

\begin{proof}
For any $i\in\N$ we have $(f_i(x))^* f_i(x)\le\sum_i (f_i(x))^* f_i(x)\le\<x,x\>$ for any $x\in\cN$. Hence $||(f_i(x))^* f_i(x)||\le||\<x,x\>||$, i.e. $||f_i(x)||^2\le||x||^2$ and $||f_i(x)||\le||x||$, so $||f_i||=\sup\limits_{||x||\le1}||f_i(x)||\le||x||=1$.
\end{proof}

\begin{rk}
If $\cN^0\ne\cN$ this is not true in general even for $\cN^0$-admissible case. Indeed, for any (non-trivial) modules $\mathcal Z_1$ and $\mathcal Z_2$ take $\cN=\cZ_1\oplus\cZ_2$, $\cN^0=\cZ_1\oplus\{0\}\subset\cZ_1\oplus\cZ_2$ and $X=\{x_1,0,\dots\}$, where $x_1=(0,w)$, $||w||>1$. Then for any $x=(z,0)\in\cN^0$ we have that $\<x,x_1\>=0$ so $X$ is $\cN^0$-admissible, but $||x_1||>1$.

It turns out that
this question is closely related to the following problem attracted attention recently.
Suppose that $\cN_0 \ss \cN$ is a Hilbert $C^*$-submodule such that its orthogonal complement is trivial:
$(\cN_0)^\bot_\cN=\{0\}$. Is it true that the norm of $x\in \cN$ is equal to its norm as an $\A$-functional on
$\cN_0$, i.e. $\sup \{|\<x,y\>| \colon y\in \cN_0, \: \|y\|\le 1\}$? The answer generally is ``no'' (see \cite{KaadSkeide2023,Manuilov2023MathNachr})
but in some cases, for example for $\A$ being a commutative von Neumann algebra, the answer is ``yes''
\cite{Manuilov2023MathNachr}.
\end{rk}

Denote by $\F$ a countable collection $\{\f_1,\f_2,\dots\}$
of states on $\A$. For each pair $(F,\F)$ 
with a $*$-$\cN^0$-admissible $F$,
consider the following seminorms

\begin{equation}\label{eq:defnnorm*}
\nu_{F,\F}(x)^2:=\sup_k 
\sum_{i=k}^\infty |\f_k\left(f_i (x)\right)|^2,\quad x\in \cN^0. 
\end{equation}

and corresponding pseudo-metrics
\begin{equation}\label{eq:defnmetr*}
d_{F,\F}(x,y)^2:=\sup_k 
\sum_{i=k}^\infty |\f_k\left(f_i (x-y)\right)|^2,\quad x,y\in \cN^0. 
\end{equation}

Let us observe that this is indeed a generalization of the seminorms $\nu_{X,\F}$ defined by (\ref{eq:defnnorm})
since $\cN$ is included in $\cN'$ by formula $\widehat{x}(z)=\<x,z\>$ and since $\<x,z\>=\<z,x\>^*$ and $\f(a^*)=\overline{\f(a)}$ for any $a\in\A$ and any state $\f$ on $\A$, so $|\f(\<x,x_i\>)|=|\f(\<x_i,x\>^*)|=|\f(\<x_i,x\>)|=|\f(\widehat{x}_i(x))|$, where $\widehat{x}_i\in\cN$ (note that in \cite{ManTroit2023} the inclusion is $\cN\to\cN'$ described by another order in the $\A$-inner product but this difference in not significant).

Note that  $\widehat{xa}(z)=\<xa,z\>=a^*\<x,z\>$. This corresponds to the structure of a right $\A$-module on $\cN'$, namely, $fa(z)=a^*f(z)$.

First, remark that this is a finite non-negative number.
Indeed, by Lemma \ref{lem:comparkvadr}
\begin{multline*}
\sum_{i=k}^s|\f_k\left(f_i (x-y)\right)|^2  \le 
  \f_k\left(\sum_{i=k}^s (f_i (x-y))^* f_i (x-y)\right) \\
\le \left\|\sum_{i=k}^s  (f_i (x-y))^* f_i (x-y) \right\|  
\le  \| \<x-y,x-y\>\| =  \|x-y\|^2. 
\end{multline*}
Since in (\ref{eq:defnmetr*}) we have a series of non-negative
numbers, this estimation implies its convergence and the estimation
\begin{equation}\label{eq:sravnsobych}
d_{F,\F}(x,y)\le  \|x-y\|.
\end{equation}

As it was noted in \cite[Proposition 2.8]{Troitsky2020JMAA}, 
for $x\ne y$ there exists $(X,\F)$ such that 
$d_{X,\F}(x,y)> \frac{1}{2} \|x-y\|$ so this uniform structure is Hausdorff.

Let us verify the triangle inequality:
\begin{equation}\label{eq:triangle_inequa_1}
\nu_{F,\F}(z+x)\le \nu_{F,\F}(z) + \nu_{F,\F}(x).
\end{equation}
Take an arbitrary $\e>0$ and choose $k$ and $m$ such that
\begin{equation}\label{eq:triangle_inequa_3}
\nu_{F,\F}(z+x)<\sqrt{\sum_{i=k}^m |\f_k (f_i(z+x))|^2}
+\e. 
\end{equation}
We have
\begin{equation}\label{eq:triangle_inequa_4}
\sqrt{\sum_{i=k}^m |\f_k (f_i(z+x))|^2}\le 
\sqrt{\sum_{i=k}^m (|\f_k (f_i(z))|+|\f_k (f_i(x)|)^2}
\end{equation}
By the triangle inequality for 
the standard norm in $\C^{m-k+1}$ we have
\begin{multline*}
\sqrt{\sum_{i=k}^m (|\f_k (f_i(z))|+|\f_k (f_i(x)|)^2}\le
\sqrt{\sum_{i=k}^m |\f_k (f_i(z))|^2}
+\sqrt{\sum_{i=k}^m |\f_k (f_i(x))|^2}\\
\le 
\nu_{F,\F}(z) + \nu_{F,\F}(x).
\end{multline*}
 Since $\e$ in (\ref{eq:triangle_inequa_3})  is arbitrary, together with (\ref{eq:triangle_inequa_4}) 
the last estimation gives (\ref{eq:triangle_inequa_1}).
So, we have verified that $d_{F,\F}$ satisfy the
conditions for
seminorms and
now can define a uniform structure on the unit ball of
$\cN^0$. 

\begin{dfn}\label{dfn:totbaundset*}
A set $Y\subseteq \cN^0 \subseteq  \cN$ is \emph{totally bounded}
with respect to this uniform structure if for any $(F,\F)$, 
where $F \subseteq  \cN'$ is $*$-$\cN^0$-admissible,
and any 
$\e>0$ there exists a finite collection $y_1,\dots,y_n$
of elements of $Y$ such that the sets
$$
\left\{ y\in Y\,|\, d_{F,\F}(y_i,y)<\e\right\}
$$  
form a cover of $Y$. This finite collection is an 
$\e$\emph{-net in $Y$ for} $d_{F,\F}$.

If so, we will say briefly that $Y$ is $(\cN,\cN^0)^*$-\emph{totally bounded}.

We will consider several variants of the definition.  If we consider only adjointable functionals $f_i$
we will say that $Y$ is $(\cN,\cN^0)^*_{ad}$-\emph{totally bounded}.
The corresponding uniform structure we will call \emph{adjointable}.

We will call a functional $f:\cN \to \A$ \emph{locally adjointable,} if for any adjointable $\A$-linear operator
$g: \A \to \cN$, the composition $f\circ g: \A \to \A$ is an adjointable functional. 
Of course, any adjointable functional is locally adjointable.
If we consider only locally adjointable functionals $f_i$
we will say that $Y$ is $(\cN,\cN^0)^*_{lad}$-\emph{totally bounded}.
The corresponding uniform structure we will call \emph{locally adjointable}.
\end{dfn}

Another special case of a functional is one given by a multiplier.
For any $C^*$-algebra $\A$ one can define the $C^*$-algebra of \emph{multipliers} $M(\A)$ (see \cite[\S3.12]{Pedersen} for details).
Then, for any Hilbert $\A$-module $\cN$ there exists a Hilbert $M(\A)$-module $M(\cN)$ (which is called the \emph{multiplier module} of $\cN$) containing $\cN$ as an ideal submodule associated with $\A$, i.e. $\cN=M(\cN)\A $ (see \cite{BakGul2004} for more details). 
Any (modular) \emph{multiplier} $m\in M(\cN)$ represents an $\A$-functional $\widehat{m}\circ j_M$ on $\cN$ by the formula $\widehat{m}(x)=\<m,x\>$. This functional is adjointable and its adjoint is given by the formula $\widehat{m}^*(a)=ma$. In fact this map gives rise to an identification of $M(\cN)$ and the module $\cN^*$ of
adjointable functionals on $\cN$ (see, \cite{BakGul2004,Bak2019}).
So, if in the definition of the uniform structure we consider only multipliers (if so, corresponding uniform structure we will call \emph{outer} (in the spirit of \cite{AramBa}) and in this case we will say that $Y$ is $(\cN,\cN^0)^*_{out}$-\emph{totally bounded}), this uniform structure coincide with $(\cN,\cN^0)^*_{ad}$-uniform structure.
In these terms it was previously introduced in \cite{Fuf2021faa}.

If $f\in\cN'$ is represented by some element of multiplier module we will denote it also as $f\in M(\cN)$ if it does not cause a confusion.

\begin{lem}\label{lem:directsum_totbu}
Suppose, a set $Y\subseteq  \cN=\cN_1\oplus \cN_2$ 
is $(\cN,\cN^0)^*$-totally bounded.
Then
$p_1 Y$ and $p_2 Y$ are 
$(\cN_1,\cN^0_1)^*$- and $(\cN_2,\cN^0_2)^*$-totally bounded, 
respectively, where 
$p_1: \cN_1\oplus \cN_2 \to \cN_1$,
$p_2: \cN_1\oplus \cN_2 \to \cN_2$ are the orthogonal projections, $\cN^0_1=p_1 \cN^0$ and $\cN^0_2=p_2 \cN^0$.

Conversely, if $p_1 Y$ and $p_2 Y$ are respectively
$(\cN_1,\cN^0_1)^*$- and $(\cN_2,\cN^0_2)^*$-totally bounded
for some submodules $\cN^0_1$ and $\cN^0_2$, 
respectively, then $Y$ is $(\cN,\cN^0_1\oplus\cN^0_2)^*$-totally
bounded. 

This is similarly true for $(\cN,\cN^0)^*_{ad}$ and $(\cN,\cN^0)^*_{lad}$-totally boundedness.
\end{lem}

\begin{proof}
Denote by $J_j=p_j^*$ the corresponding inclusions
$J_j:\cN_j\hookrightarrow \cN_1\oplus \cN_2$, $j=1,2$.

Also we introduce the map  $p_j':\cN_j'\to (\cN_1\oplus \cN_2)'$ defined by the formula $p_j'(f)(x)=f(p_j(x))$, and the map $J_j': (\cN_j\oplus \cN_2)' \to \cN_j'$, $J_j'(f)(x)=f(J_j(x))$.

Suppose, $Y$  is $(\cN,\cN^0)^*$-totally bounded and $F=\{f_i\}$ is 
a $*$-admissible system for a 
submodule
$\cN^0_1\subseteq  \cN_1$. Then $p_1'F=\{p_1'(f_i)\}$ is admissible for $\cN^0$ because
$$
(p_1'(f_i)(x))^*p_1'(f_i)(x)=(f_i(p_1(x)))^*f_i(p_1(x)).
$$
Let $y_1,\dots,y_s$ be an $\e$-net in $Y$ for $d_{p_1'F,\F}$.
Then $p_1 y_1,\dots,p_1 y_s$ is an $\e$-net in $p_1 Y$ for 
$d_{F,\F}$.
Indeed, consider an arbitrary $z\in p_1 Y$. Then $z=p_1 y$ for
some $y\in Y$. Find $y_k$ such that $d_{p_1'F,\F}(y,y_k)<\e$.
Then
\begin{multline*}
d_{F,\F}^2(z,p_1 y_k)=\sup_k 
\sum_{i=k}^\infty |\f_k\left(f_i(z-p_1 y_k)\right)|^2 \\
= \sup_k
\sum_{i=k}^\infty |\f_k\left(f_i(p_1(y- y_k))\right)|^2
\\
=
\sup_k \sum_{i=k}^\infty |\f_k\left((J_1'f_i)(y- y_k)\right)|^2=
d_{p_1' F,\F}^2(y,y_k)<\e^2.
\end{multline*}
Similarly for $j=2$.

Conversely, suppose that $p_j Y$ are 
$(\cN_j,\cN_j^0)^*$-totally bounded, $j=1,2$.
Let $F=\{f_i\}$ be an admissible system in $\cN'$ for 
$\cN^0_1\oplus \cN^0_2$ and $\e>0$ is arbitrary.
Then $F_j:=\{J_j'(f_i)\}$ is an admissible system in $\cN_j'$
for $\cN^0_j$. Indeed, this follows from
$$
(J'_j (f_i)(x))^*J'_j (f_i)(x)=  (f_i(J_j(x)))^*f_i(J_j(x)).
$$
For $u, v\in p_j Y$, we have
\begin{equation}\label{eq:projprojproj}
d_{F,\F} (J_j  u, J_j v) = 
d_{F_j,\F} ( u, v).
\end{equation}
Indeed, we obtain the convergence and can estimate the sum
using (as above) the equality
$$
\sum_{i=1}^s (f_i(J_j(u-v)))^*f_i(J_j(u-v))
= \sum_{i=1}^s ((J_j'f_i)(u-v))^*(J_j'f_i)(u-v)
$$
and, quite similarly, (\ref{eq:projprojproj}) follows from the equality
$$
f_i(J_j  u- J_j v)= (J_j'f_i)(u- v),\qquad j=1,2.
$$

Suppose that $z_1,\dots,z_m$ is an $\e/4$-net 
in $p_1 Y$ for $d_{F_1,\F}$ and
$w_1,\dots,w_r$ is an $\e/4$-net 
in $p_2 Y$ for $d_{F_2,\F}$. 
Consider $\{z_k+w_s\}$, $k=1,\dots,m$,
$s=1,\dots,r$. Then $\{J_1 z_k+ J_2 w_s\}$ is an
$\e/2$-net in $p_1 Y \oplus p_2 Y$
 for $d_{F,\F}$.
Indeed, for any $J_1 p_1 y_1 +J_2 p_2 y_2$, $y_1,y_2\in Y$,
one can find $z_k$ and $w_s$
such that
$$
d_{F_1,\F}(p_1 y_1,z_k)<\e/4,\qquad
d_{F_2,\F}(p_2 y_2,w_s)<\e/4.
$$
Then by (\ref{eq:triangle_inequa_1}) and (\ref{eq:projprojproj})
\begin{multline*}
d_{F,\F}(J_1 p_1 y_1 +J_2 p_2 y_2, J_1 z_k+ J_2 w_s) \\
\le d_{F,\F}(J_1 p_1 y_1,J_1 z_k)+
d_{F,\F}(J_2 p_2 y_2, J_2 w_s) \\
=d_{F_1,\F}(p_1 y_1, z_k)+
d_{F_2,\F}(p_2 y_2,  w_s)<\e/2.
\end{multline*}
Now find a subset $\{u_l\}\subset \{z_k+w_s\}$
formed by all elements of $\{z_k+w_s\}$ such that there exists an
element $u^*\in Y\sse p_1 Y \oplus p_2 Y$ 
with $d_{F,\F}(u^*,z_k+w_s)<\e/2$.  
Denote these $u^*$ by $u^*_l$, $l=1,\dots,L$. So,
\begin{enumerate}[1)]
\item for any $y\in Y$, there exists $l\in 1,\dots, L$
such that $d_{F,\F}(y,u_l)<\e/2$;
\item for each $l\in 1,\dots, L$, we have $d_{F,\F}(u^*_l,u_l)<\e/2$.
\end{enumerate}
By the triangle inequality, $\{u^*_l\}$ is a finite $\e$-net
in $Y$ for $d_{F,\F}$ and we are done.
\end{proof}

\begin{rk}
From \cite[Lemma 2.15]{Troitsky2020JMAA} it follows that $\cN^0_1\oplus\cN^0_2$ is countably generated
if and only if $\cN^0_1$ and $\cN^0_2$ are countably generated.
\end{rk}

Evidently we have the following statements.
\begin{prop}
A functional $f:\A\to \A$ is locally adjointable if and only if it is adjointable.
\end{prop}

\begin{prop}
If $LM(\A)=M(\A)$, hence $RM(\A)=M(\A)$, then any functional is locally adjointable.
\end{prop}

This is fulfilled, in particular, for commutative and unital algebras.

Now we are able to formulate our main result.

\begin{teo}[Main Theorem]\label{teo:mainteorem}
Let $Y$ be a subset of $\cN^0\subset\cN$. 
Then $Y$ is $(\cN,\cN^0)^*_{lad}$-totally bounded
if and only if it is $(\cN,\cN^0)$-totally bounded in the following cases:
\begin{enumerate}
\item[\rm 1)] $\cN=\A$;
\item[\rm 2)] $\cN=\ell^2(\A)$;
\item[\rm 3)] $\cN$ is countably generated.
\end{enumerate}
\end{teo}

\begin{cor}\label{cor:main1}
Suppose that $LM(\A)=M(\A)$ and $Y$ is a subset of $\cN^0\subset\cN$.
Then $Y$ is $(\cN,\cN^0)^*$-totally bounded
if and only if it is $(\cN,\cN^0)$-totally bounded
and if and only if it is $(\cN,\cN^0)^*_{ad}$-totally bounded
in the following cases:
\begin{enumerate}
\item[\rm 1)] $\cN=\A$;
\item[\rm 2)] $\cN=\ell^2(\A)$;
\item[\rm 3)] $\cN$ is countably generated.
\end{enumerate}
\end{cor}

From Theorem \ref{teo:mainteoremold} and Theorem \ref{teo:mainteorem} we immediately obtain

\begin{teo}\label{teo:mainteorem1}
Suppose, $F:\M\to\cN$ is an adjointable operator and
$\cN$ is countably generated. Then $F$ is $\A$-compact
if and only if $F(B)$ is $(\cN,\cN)^*_{lad}$-totally bounded,
where $B$ is the unit ball of $\M$.
\end{teo}

A similar 
statement based on Corollary \ref{cor:main1} can be obtained as well.

For outer systems we will be able to prove this statement (see Theorem \ref{teo:main_out} below) without the
countability restriction. This is not surprising because $\cN^*$ is rather close to $\cN$ (and coincides with it, e.g. for a unital $\A$).

The following fact is very useful.

\begin{lem}\label{new_adm}
Suppose that $F=\{f_{i}\}$ is a $*$-$\cN^0$-admissible system and $\{g_i\}_{i\in\N}$ is an arbitrary sequence of elements of $\A$ such that $||g_i||\le1$ for any $i$. Then $\{(f_{i}g_i)\}$ is also a $*$-$\cN^0$-admissible system.
Also, if $F$ is outer $\cN^0$-admissible system (i.e. all $f_i\in M(\cN)$), then $\{(f_{i}g_i)\}$ is $\cN^0$-admissible (i.e. $f_{i}g_i\in\cN$)

\end{lem}

\begin{proof}
Indeed, 
\begin{multline*}
((f_{i}g_i)(x))^*(f_{i}g_i)(x)
=((g_i)^*f_{i}(x))^*((g_i)^*f_{i}(x))
=(f_{i}(x))^*g_ig_i^*f_{i}(x)
\\
\le
||g_i||^2(f_{i}(x))^*f_{i}(x)\le(f_{i}(x))^*f_{i}(x).
\end{multline*}
This implies properties 1) and 2) of Definition \ref{dfn:admissyst*}. 
The property 3) is evident.
\end{proof}

\begin{lem}\label{bounded} Suppose $Y\subset \cN^0\subset\cN$ is bounded w.r.t. any seminorm of $(\cN,\cN^0)$-uniform structure. Then $Y$ is bounded in norm.
\end{lem}

\begin{proof}
Suppose that $Y$ is not norm-bounded. Then there exists a sequence $\{z_k\}_{k\in\N}\subset\cN^0$ such that $||z_k||\ge3^k$. Define $x_k=\frac{z_k}{||z_k||2^k}$. The collection $X=\{x_k\}$ is $\cN^0$-admissible, even $\cN$-admissible, by the following inequality: $\sum\limits_k\<x,x_k\>\<x_k,x\>\le\sum\limits_k\frac{1}{4^k}\<x,x\>\le\<x,x\>$ for any $x\in\cN$. For any $k\in\N$ take a state $\f_k$ such that $\f_k(\<z_k, z_k\>)\ge\frac{1}{2}||\<z_k,z_k\>||=\frac{1}{2}||z_k||^2$. Then
$$
\nu_{X,\F}(z_k)\ge |\f_k(\<z_k,x_k\>)|=\frac{1}{||z_k||2^k}|\f_k(\<z_k,z_k\>)|\ge\frac{1}{||z_k||2^{k+1}}||z_k||^2\ge\frac{3^k}{2^{k+1}}, 
$$
i.e. $Y$ is not bounded w.r.t. the seminorm $\nu_{X,\F}$. A contradiction.
\end{proof}

\begin{cor}
Suppose $Y\subset \cN^0\subset\cN$ is bounded w.r.t. any seminorm of $(\cN,\cN^0)^*_{lad}$, $(\cN,\cN^0)^*_{ad}$ or $(\cN,\cN^0)^*$-uniform structure. Then $Y$ is bounded in norm.
\end{cor}

\begin{cor}
Suppose $Y\subset \cN^0\subset\cN$ is totally bounded w.r.t. any seminorm of $(\cN,\cN^0)$, $(\cN,\cN^0)^*_{lad}$, $(\cN,\cN^0)^*_{ad}$ or $(\cN,\cN^0)^*$-uniform structure. Then $Y$ is bounded in norm.
\end{cor}

\section{The case of multipliers}
\begin{lem}\label{approx}
For any $\A$-module $\cN$, any $f\in\cN'$, any state $\f$ on $\A$ and any $n\in\N$ there exists a positive $g_n\in\A$, $||g_n||\le1$, such that

$$
|\f(f(x))-\f((fg_n)(x))|=|\f(f(x))-\f(g_nf(x))|\le\frac{||f(x)||}{n}\le\frac{||f||\cdot||x||}{n}
$$ 
for any $x\in\cN$.
\end{lem}

\begin{proof}
From Lemma \ref{uniform} we have that for any $n\in \N$ there exists a positive $g_n\in\A$, $||g_n||\le1$, such that $|\f(y)-\f(g_ny)|\le\frac{||y||}{n}$ for any $y\in\A$. By taking $y=f(x)$ we have $|\f(f(x))-\f(g_nf(x))|\le\frac{||f(x)||}{n}\le\frac{||f||\cdot||x||}{n}$ for any $x\in\cN$, then note that $(fb)(\cdot)=b^*f(\cdot)$ for arbitrary $b\in\A$.
\end{proof}

Now, for any $\e>0$ and any $x$ such that $||x||\le d$ there exists $n(i)\in\N$ such that
$$
|\f(f(x))-\f((fg_{n(i)})(x))|\le\frac{\e}{2^i}.
$$

If $f\in M(\cN)$ (for example it is always so if $f$ is locally adjointable and $\cN=\A$ as a module over itself), then $fg_{n(i)}\in\cN$ and $fg_{n(i)}(x)=\<fg_{n(i)},x\>=g_{n(i)}\<f,x\>$.

Moreover, for every finite family of states $\{\f_j\}_{j=1}^L$ we can find $n(i)$ such that
$$
|\f_j(f(x))-\f_j((fg_{n(i)})(x))|\le\frac{\e}{2^i}
$$
for all $j=1,\dots,L$.

\begin{teo}\label{teo:main_out}
Let $Y$ be a subset of $\cN^0\subset\cN$. 
Then $Y$ is $(\cN,\cN^0)_{out}$-totally bounded
if and only if it is $(\cN,\cN^0)$-totally bounded.
\end{teo}

\begin{proof}
For any outer $\cN^0$-admissible system $F=\{f_i\}$ of multipliers of $\cN$ and any countable collection $\F=\{\f_j\}$ of states on $\A$
it is sufficient to find, for arbitrary $\e>0$,
a $\cN^0$-admissible system $X=\{x_i\}$ in $\cN$ such that for any $x\in \cN^0$
with $\|x\|\le \mathrm{diam} (Y)=:d<\infty$ we have
$$
d_{F,\F}(x,0)\le   d_{X,\F}(x,0) +\e.
$$
Indeed, this means that an $\e$-net on $Y$ for $d_{X,\F}$ is a $2\e$-net for $d_{F,\F}$.
We may consider
\begin{equation}\label{eq:epsi2}
\e<1, \mbox{ hence }\e^4<\e^2.
\end{equation}
By Lemma \ref{uniform}, for each $i=1,2,\dots$ 
we can find 
$g_{n(i)}\ge 0$ such that, for all $x\in Y$, we have
\begin{equation}\label{eq:main_estim2}
|\f_k(\<f_i,x\>)-\f_k(g_{n(i)}\<f_i,x\>)|\le\frac{\e^2}{2^i\cdot4\max\{1,d\}} ,\qquad k=1,\dots,i,
\end{equation}
i.e. $x_i= f_i g_{n(i)} \in \cN$.  
The system $\{x_i\}$ is $\cN^0$-admissible due to Lemma \ref{new_adm}.
For arbitrary $x\in Y$ and arbitrary $k,i\in \N$, $k\le i$, we have
$$
|\f_k(\<f_i,x\>)|\le|\f_k(\<f_i,x\>)-\f_k(\<x_i,x\>)|+|\f_k(\<x_i,x\>)|. 
$$
Hence, by (\ref{eq:main_estim2}) and (\ref{eq:epsi2})
\begin{multline*}
|\f_k(\<f_i,x\>)|^2 
\le|\f_k(\<f_i,x\>)-\f_k(\<x_i,x\>)|^2\\ +2|\f_k(\<f_i,x\>)-\f_k(\<x_i,x\>)||\f_k(\<x_i,x\>)|
+|\f_k(\<x_i,x\>)|^2\\
\le\frac{\e^2}{2^i\cdot4}+2\frac{\e^2}{2^i\cdot4d }d+|\f_k(\<x_i,x\>)|^2\le\frac{\e^2}{2^i}+|\f_k(\<x_i,x\>)|^2.
\end{multline*}
Then
$$
\sum\limits_{i=k}^\infty|\f_k(\<f_i,x\>)|^2\le\e^2+\sum\limits_{i=k}^\infty|\f_k(\<x_i,x\>)|^2.
$$
Thus, using  $\sqrt{s+t}\le \sqrt{s+ 2\sqrt{st} +t}=\sqrt{s}+\sqrt{t}$, for $s,t\ge 0$, we obtain
$$
\sqrt{\sum\limits_{i=k}^\infty|\f_k(\<f_i,x\>)|^2}\le\e+\sqrt{\sum\limits_{i=k}^\infty|\f_k(\<x_i,x\>)|^2}.
$$
Taking at first the supremum on the right hand side and then on the left hand side, we obtain 
$$
\sup\limits_k\sqrt{\sum\limits_{i=k}^\infty|\f_k(\<f_i,x\>)|^2}\le\e+\sup\limits_{k}\sqrt{\sum\limits_{i=k}^\infty|\f_k(\<x_i,x\>)|^2},
$$
i.e. $d_{F,\F}(x,0)\le   d_{X,\F}(x,0) +\e$ as desired.
\end{proof}

\section{Proof of the main theorem}

\subsection{Proof of the main theorem for $\cN=\A$}
This proof is quite similar to the previous one.
From Lemma \ref{approx} we have the following as a corollary.

\begin{teo}\label{teo:approx_via_au}
Suppose, that $\A$ be a $C^*$-algebra
with an approximate unit $\omega_\alpha$, $\f$ is a state on $\A$, $f$ is a (locally) adjointable 
$\A$-linear functional 
(i.e. $f\in M (\A)$, $f(a)=f^*a$) and $\e>0$ an arbitrary number. Then, for a sufficiently large $\alpha$,
and any $a \in \A$ ,
$$
|\f((1-\omega_\alpha) f (a))| 
=|\f(   (f^*-(f\omega_\alpha)^*)  a)|
< \e \|a\|.
$$ 
\end{teo}

Now let $F=\{f_i\}$ be a $*$-admissible system of locally adjointable (=adjointable) functionals for $\A$ and $\F=\{\f_j\}$ a countable collection
of states on $\A$. 
As in the previous section, to prove the main theorem for $\cN=\A$ it is sufficient to find, for arbitrary $\e>0$ (again suppose that $\e<1$),
an admissible system $X=\{x_i\}$ in $\A$ such that for any $a\in \A$
with $\|a\|\le \mathrm{diam} (Y)=:d<\infty$ we have
$$
d_{F,\F}(a,0)\le   d_{X,\F}(a,0) +\e.
$$
By Theorem \ref{teo:approx_via_au}, for each $i=1,2,\dots$ and the above described $a$, we can find 
$\a(i)$ such that for $j=1,\dots,i$,
\begin{equation}\label{eq:main_estim}
|\f_j((f_i^* - \omega_{\a(i)} f_i^* )a)|
=|\f_j((f_i^* - ( f_i\omega_{\a(i)})^* )a)|
<\frac 1{16 \max\{1,d\}} \cdot \frac{\e^2}{2^i},
\end{equation}
i.e. $x_i^*= (f_i \omega_{\a(i)})^* \in \A$ (this is possible only if $f_i \in M(\A)$, not only $RM(\A)$).
Then, for arbitrary $a\in \A$, choose a number $k$ such that
\begin{flalign*}
d_{F,\F}(a,0)\le& \left(\sum_{i=k}^\infty |\f_k(f_i^* a)|^2\right)^{1/2} + \frac \e 2 \\
\le & \left(\sum_{i=k}^\infty |\f_k(f_i^* a - x_i^* a)+ \f_k(x_i^*a)|^2\right)^{1/2} + \frac \e 2 \\
\end{flalign*}
$$
=\left(\sum_{i=k}^\infty |\f_k(f_i^* a - x_i^* a)|^2+ 2 |\f_k(f_i^* a - x_i^* a)|\cdot  |\f_k(x_i^*a)|+  |\f_k(x_i^*a)|^2\right)^{1/2} + \frac \e 2 
$$
\begin{flalign*}
\le & \left(\sum_{i=k}^\infty \left(\left(\frac 1 {16} \cdot \frac{\e^2}{2^i}\right)^2+ \frac 1 8 \cdot \frac{\e^2}{2^i}\right)+  (d_{X,\F}(a,0))^2\right)^{1/2} + \frac \e 2  \qquad\qquad \text{(by (\ref{eq:main_estim})}\\  
\le & \left(\frac 14 \sum_{i=k}^\infty   \frac{\e^2}{2^i} +  (d_{X,\F}(a,0))^2\right)^{1/2} + \frac \e 2 
 \qquad\qquad \text{(by 
 the condition $\e<1$
 )}\\
= &   \left(\frac{\e^2}{4} +  (d_{X,\F}(a,0))^2\right)^{1/2} + \frac \e 2 \le
d_{X,\F}(a,0)+\e, 
\end{flalign*}
because, as above, $\sqrt{s+t}\le 
\sqrt{s}+\sqrt{t}$, for $s,t\ge 0$.

From Lemma \ref{new_adm} it follows that $X=\{x_i\}$ is an admissible system.

\subsection{Proof of the main theorem for $\cN=\ell^2(\A)$}
To reduce this case to the case of $\cN=\A$, recall the description of arbitrary functional $f:\ell^2(\A) \to \A$ from \cite[Theorem 2.3]{Bak2019}.

In that paper the inner product is defined to be anti-linear on second variable, 
and embedding $\ell^2(\A)\to\ell^2(\A)'$ is defined by the formula $x\mapsto\widehat{x}(\cdot)=\<\cdot,x\>$, so it is extended for left multipliers $x$ by the formula $\widehat{x}(y)=\sum\limits_{s=1}^\infty y_sx^*_s$
as an isometric isomorphism $\ell^2_{strong}(LM(\A))\to\ell^2(\A)'$.  Hence $x$ needs to satisfy the condition $\sup\limits_N \left\|\sum\limits_{s=1}^Nx_sx_s^* \right\|<\infty$.

In our case we define the inner product to be anti-linear on first variable and define embedding $\ell^2(\A)\to\ell^2(\A)'$ by the formula $x\mapsto\widehat{x}(\cdot)=\<x,\cdot\>$, hence after extension the embedding in our case $x$ needs to be a a sequence of right multipliers acting by the formula $\widehat{x}(y)=\sum\limits_{s=1}^\infty x^*_sy_s$ and it must satisfy the condition $\sup\limits_N\left\|\sum\limits_{s=1}^Nx_s^*x_s \right\|<\infty$.

So, any functional $f:\ell^2(\A) \to \A$ can be described
as a sequence $f_s \in RM(\A)$, $s=1,2,\dots$, such that the partial sums of the series $\sum_s f_s^* f_s$ are
uniformly bounded.
If $f$ is locally adjointable, then $f_s \in M(\A)$, $s=1,2,\dots$, because the inclusion of $\A$ into $\ell^2(\A)$
(as the $s^{th}$ summand) is adjointable.

Since the partial sums of the series $\sum_s f_s^* f_s$ are
uniformly bounded, the series $\sum_s \f(f_s^* f_s)$ is convergent, where $\f$ is an arbitrary state on $\A$.
Here we consider any state as a state on $M(\A)$ due to \cite[2.3.24]{BrRob}.

By Lemma \ref{bounded} there exists $d<\infty$ such that $\|x\|\le d$ for any $x\in Y$.

So, for a $*$-$\cN^0$-admissible system $F=\{f_i\}$, $f_i=(f_{i,1}, f_{i,2},\dots)$ we first choose a finite part
of each functional $(f_{i,1}, f_{i,2},\dots,f_{i,r(i)},0,\dots)$ such that
\begin{equation}\label{eq:first_for_ell2}
\sum_{s=r(i)+1}^\infty \f_k(f_{i,s}^* f_{i,s}) < \frac{\e^2}{2^i\cdot4d^2} \,\,  \text{for all} \,\, k=1,\dots,i.
\end{equation}

Note that for any $k,i,p,q\in\N$, $k\le i$, $r(i)\le p\le q$, the function $(a,b)\mapsto\f_k(\sum\limits_{s=p}^qa_s^*b_s)$ is a complex inner product on $\ell^2_{strong}(M(\A))$, so by the Cauchy-Schwartz inequality we have
\begin{eqnarray*}
\left|\f_k\left(\sum\limits_{s=p}^qf_{i,s}^*x_s\right)\right|^2&\le&
{\f_k\left(\sum\limits_{s=p}^qf_{i,s}^*f_{i,s}\right)\f_k\left(\sum\limits_{s=p}^qx_s^*x_s\right)}\\
&\le& {\f_k\left(\sum\limits_{s=r(i)+1}^\infty f_{i,s}^*f_{i,s}\right)||x||^2}\le\frac{\e^2d^2}{2^i\cdot4d^2}
=\frac{\e^2}{2^i\cdot4}.
\end{eqnarray*}

Second, we approximate each multiplier $f_{i,s}$, $s=1,\dots,r(i)$, by choosing $g_{n(i)}$ for $f_{i,s}g_{n(i)}$ as in the previous section (by  Lemma \ref{uniform}):
\begin{equation}\label{eq:sec_for_ell2}
|\f_k(f_{i,s}^*x)-\f_k(g_{n(i)}f_{i,s}^*x)|\le\frac{\e}{2^i\cdot4r(i)} ,\qquad  \text{for all} \,\, k=1,\dots,i.
\end{equation}

Hence, by using the inequality for scalars $|a+b|^2\le2|a|^2+2|b|^2$ we have for all $k\le i$
\begin{multline*}
|\f_k(f_i(x))|^2
=
\left|\f_k \left(\sum\limits_{s=1}^\infty f_{i,s}^* x_s \right)\right|^2
\\
\le
2\left|\f_k\left(\sum\limits_{s=1}^{r(i)} f_{i,s}^* x_s\right)\right|^2
+2\left|\f_k\left(\sum\limits_{s=r(i)+1}^\infty f_{i,s}^* x_s\right)\right|^2\le
\end{multline*}
\begin{multline*}
\le
4\left|\f_k\left(\sum\limits_{s=1}^{r(i)} g_{n(i)}f_{i,s}^* x_s\right)\right|^2+
4\left|\f_k\left(\sum\limits_{s=1}^{r(i)} f_{i,s}^* x_s\right)-
\f_k\left(\sum\limits_{s=1}^{r(i)} g_{n(i)}f_{i,s}^* x_s\right)\right|^2
\\+
\frac{\e^2}{2^i\cdot2}\le
\end{multline*}
$$
\le
4\left|\f_k\left(\sum\limits_{s=1}^{r(i)} g_{n(i)}f_{i,s}^* x_s\right)\right|^2+
4\frac{\e^2}{4^i\cdot16}+
\frac{\e^2}{2^i\cdot2}.
$$
Thus,
$$
\sum\limits_{i=k}^\infty|\f_k(f_i(x))|^2\le\sum\limits_{i=k}^\infty4\left|\f_k\left(\sum\limits_{s=1}^{r(i)} g_{n(i)}f_{i,s}^* x_s\right)\right|^2+\e^2.
$$
Hence, using again  $\sqrt{s+t}\le \sqrt{s+ 2\sqrt{st} +t}=\sqrt{s}+\sqrt{t}$, for $s,t\ge 0$,
we obtain
$$
\sqrt{\sum\limits_{i=k}^\infty|\f_k(f_i(x))|^2}\le2\sqrt{\sum\limits_{i=k}^\infty\left|\f_k\left(\sum\limits_{s=1}^{r(i)} g_{n(i)}f_{i,s}^* x_s\right)\right|^2}+\e.
$$
Taking at first the supremum on the right hand side and then on the left hand side, 
we obtain
$$
\sup\limits_k\sqrt{\sum\limits_{i=k}^\infty|\f_k(f_i(x))|^2}\le2\sup\limits_k
\sqrt{\sum\limits_{i=k}^\infty \left|\f_k\left(\sum\limits_{s=1}^{r(i)} g_{n(i)}f_{i,s}^* x_s\right)\right|^2}+\e,
$$
i.e. if we denote $X=\{x_i\}$, $x_i=(f_{i,1}g_{n(i)},f_{i,2}g_{n(i)},\dots,f_{i,r(i)}g_{n(i)},0,\dots )\in\ell^2(\A)$,
$$
\sup\limits_k\sqrt{\sum\limits_{i=k}^\infty|\f_k(f_i(x))|^2}\le2\sup\limits_k\sqrt{\sum\limits_{i=k}^\infty|\f_k(\<x_i, x\>)|^2}+\e.
$$
Thus,
$$
d_{F,\F}(x,0)\le2d_{X,\F}(x,0)+\e,
$$
for any $x\in Y$, i.e. an $\e$-net on $Y$ for $d_{X,\F}$ is a $3\e$-net for $d_{F,\F}$.
It is easy to see that $X$ is $(\cN,\cN^0)$-admissible. 

\subsection{Proof of the main theorem in the general case}
The general case of a countably generated module $\cN$ can be reduced to the above considered case of $\cN=\ell^2(\A)$ using the
Kasparov stabilization theorem
by considering the module as a direct summand of the standard one and using the fact that uniform structures respect direct summand decomposition. More specifically, denote by $S$ the map $S:\cN\to\cN\oplus\ell^2(\A)\cong\ell^2(\A)$ given by the Kasparov theorem. 
Suppose that $Y$ is $(\cN,\cN^0)$-totally bounded.
Then by Lemma \ref{lem:directsum_totbu} it follows that $S(Y)$ is $(\ell^2(\A),S(\cN^0))$-totally bounded. So, by previous section $S(Y)$ is $(\ell^2(\A),S(\cN^0))^*_{lad}$-totally bounded, hence, again by Lemma \ref{lem:directsum_totbu} we have that $Y$ is $(\cN,\cN^0)^*_{lad}$-totally bounded.
The converse statement is obvious.

\end{document}